\documentclass[12pt]{article}
\usepackage{comment}
\usepackage{natbib}
\usepackage{amssymb,latexsym}
\usepackage{graphicx,amsmath,amsfonts,amsthm}
\usepackage{color}
\usepackage{graphicx}
\usepackage{etex}

\input prepictex
\input postpictex
\input pictexwd 

\newcommand{\R}{\mathbb R}
\newcommand{\N}{\mathbb N}

\newcommand{\bpic}[4]{\beginpicture
  \setcoordinatesystem units  <1pt,1pt>
  \setplotarea x from #1 to #2, y from #3 to #4}
\newcommand{\epic}{\endpicture}

\newcommand{\bull}[2]{\put{$\bullet$} at #1 #2 }

\newtheorem{theorem}{Theorem}
\newtheorem{example}{Example}

\newtheorem{note}{Note}

\newcommand{\cd}{{\cal D}}

\title{\bf A hidden Condorcet domain in Loday's realisation of the associahedron}
\author{Arkadii Slinko}

\date{}

\begin{document}

\maketitle

\begin{abstract}
We prove that Loday's polytopal realisation of the $n$th Tamari lattice, called associahedron, has $2^{n-1}$  common points with the permutohedron, which form a maximal never-middle (symmetric) Condorcet domain.
\end{abstract}

\section{Introduction}

The origin of Condorcet domains is in voting theory. A Condorcet domain is a set of linear orders on a given set of alternatives such that, if every voter is known to have a preference from that set, the pairwise majority relation is acyclic (see \cite{puppe2024maximal} for a recent survey). Condorcet domains viewed as a set of permutations have deep relations to various branches of combinatorics \citep{GR:2008,DKK:2012,labbe2020cambrian,slinko2024combinatorial}.
In this note we present one more combinatorial connection of Condorcet domains formulated in the abstract.

\cite{DanilovK13} introduced one operation on Condorcet domains which we will use here. For any two Condorcet domains $\cd_1, \cd_2$ we set
\[
\cd_1\star \cd_2 =\{u_1u_2 \mid u_1\in \cd_1, u_2\in \cd_2 \}\cup \{u_2u_1 \mid u_1\in \cd_1,\ u_2\in \cd_2 \}
\]
is the operation used in \cite{DanilovK13}. The composition allowed them to construct a series of maximal Condorcet domains, namely,
\[
\cd=(\ldots ((a_1\star a_2)\star a_3)\star\ldots)\star a_n,
\]
where $a_i\in [n]$ is identified with the trivial domain on a single alternative $a_i$. The latter domain is known \citep{karpov2023symmetric} that such domain satisfies for each triple $i<j<k$ the never condition $kN_{\{i,j,k\}}2$ (never-middle one) or in terms of pattern avoidance $\cd$ includes all permutations of $S_n$ that avoid patterns $132$ and $231$.

\section{Permutations and binary trees.}

The group of permutations $S_n$ of order $n$ is in a bijection with the set of full binary trees with levels on leaves labelled $0,1,\ldots,n$. Every internal vertex is assigned a level as shown on Figure~\ref{fig:tree_levels}. 
\begin{figure}[h]
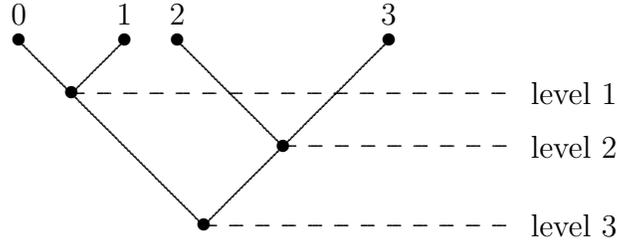

\centering
$$\bpic{-50}{50}{0}{80}
\bull{0}{0}
\bull{30}{30}
\bull{-50}{50}
\bull{-70}{70}
\put{$0$} at -70 80
\bull{-30}{70}
\put{$1$} at -30 80
\bull{70}{70}
\put{$3$} at 70 80
\bull{-10}{70}
\put{$2$} at -10 80
\setlinear
\plot 30 30 0 0 -50 50 -70 70 /
\plot -50 50 -30 70 /
\plot 70 70 30 30 -10 70 /
\setdashes
\plot 30 30 120 30 /
\plot -50 50 120 50 /
\plot 0 0 120 0 /
\put{level $1$} at 140 50
\put{level $2$} at 140 30
\put{level $3$} at 140 0
\epic$$
\caption{\label{fig:tree_levels} Binary tree with levels}
\end{figure}
A full binary tree with levels corresponds to the permutation 
\[
\sigma(i)=\text{the lowest level on the unique path from $i-1$ to $i$}.
\]
For example, the tree on Figure~\ref{fig:tree_levels} corresponds to the transposition $(2,3)$.

Let $Y_n$ be the set of full binary trees with $n + 1$ leaves $0,1,\ldots,n$:
\begin{center}
\includegraphics[height=2cm]{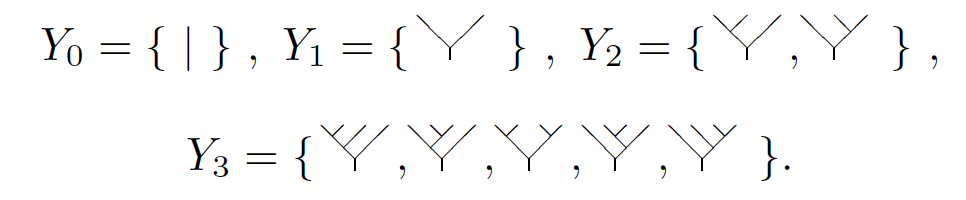}
\end{center}
Using bijection of $S_n$ with trees with levels and forgetting the levels we obtain a surjective mapping
\[
\psi\colon S_n\to Y_n.
\] 
For example,
\[
\psi((1,3,2))=\includegraphics[height=7mm]{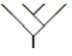}.
\]
An $n$th Tamari lattice $T_n$ \citep{tamari1962algebra}, is a partially ordered set in which the elements consist of different ways of parenthesization of a product of $n$ elements in an associative algebra. For $n=4$ we have five ways to calculate the product $abcd$, namely,
\[
((ab)c)d=(ab)(cd)=(a(bc))d= a((bc)d)= a(b(cd)).
\]
One parenthesization is ordered before another if the second parenthesization may be obtained from the first by only rightward applications of the associative law replacing $(xy)z$ with $x(yz)$. For example, the fourth Tamari lattice $T_4$ is
\begin{center}
\includegraphics[height=6cm]{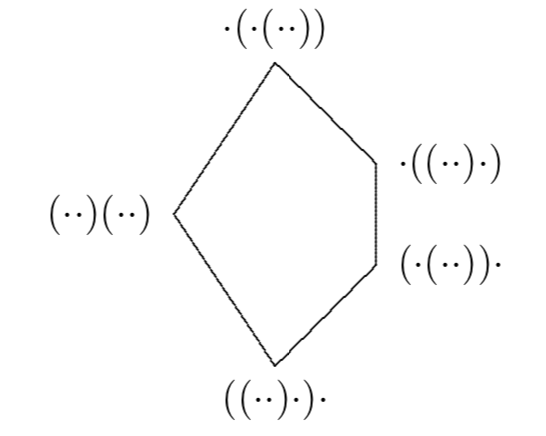}
\end{center}
So $Y_n$ can be obviously identified with the elements of the Tamari lattice. Indeed, suppose we have a bijections  $\mu_i\colon Y_i\to T_i$ for $i=1,\ldots,n-1$. Then given a tree 
$t\in T_n$ we consider its left and right branches $t_\ell$ and $t_r$, with $k$ and $n-k$ leaves, respectively, and put in correspondence to $t$ the parenthesization $(\mu_k(t_\ell))(\mu_{n-k}(t_r))$. 
This bijection induces the structure of a lattice to $Y_n$. 

\section{Permutohedron and associahedron}

We view a permutation $u\in S_n$ 
as a vector $(u(1),u(2),\ldots, u(n))\in \R^n$ and sometimes as a sequence $u(1)u(2)\ldots u(n)$. Geometrically, the permutations of $[n]$ form a set of $n!$ points in $\N^n$.
According to \cite{ziegler2012lectures} the {\em permutohedron} of order $n$ is an $(n - 1)$-dimensional polytope embedded in an $n$-dimensional space. Its vertex coordinates (labels) are the permutations of $[n]$, it is denoted as $\text{Perm}_n$.

The skeleton of a polytope is the graph structure given by the vertices and edges of the polytope. Given a graph $G$, its {\em polytopal realisation} is a polytope $P$ whose skeleton is isomorphic to $G$. 

{\em Associahedron} is a polytopal realisation of the Tamari lattice. Associahedra are also called Stasheff polytopes \citep{stasheff1963homotopy} after the work of Jim Stasheff who was the first to give their polytopal realisations.  The following beautiful polytopal realisation is due to \cite{loday2004}.
As an element of the Tamari lattice $T_n$ can be identified with a full binary tree in $Y_n$ with leaves $0,1,\ldots,n$, we label the internal vertices of that tree so that the $i$th vertex is the one which falls in between the leaves $i-1$ and $i$. Then we
denote by $a_i$, respectively, $b_i$, the number of leaf descendants of the left child, respectively, right
child, of the $i$th vertex. The product $w(i)=a_ib_i$ is called the {\em weight} of the $i$th
vertex. To the tree $t$ in $Y_n$ we associate the point $M_n(t) \in \R^n$ whose $i$th
coordinate is the weight of the $i$th vertex, that is, $M_n\colon Y_n\to \R^n$ such that
\[
M_n(t)=(a_1b_1, a_2b_2,\ldots, a_nb_n)\in \R^n.
\]
\begin{theorem}[Loday, 2004]
The polytope with vertices $\{M_n(t)\mid t\in Y_n\}$ is a polytopal realisation of the Tamari lattice $T_n$, i.e., an associahedron, denoted $\text{Asso}_n$.
\end{theorem}

\begin{example}
\label{ex:vertices_T3}
For the tree $t$ on Figure~\ref{fig:tree_levels} we have $M_3(t)=(1,4,1)$ and $M_3(\includegraphics[height=5mm]{one_tree.png})=(3,1,2)$. Figure~\ref{fig:Loday_coord} shows the coordinates of $\text{Asso}_3$.
\begin{center}
\begin{figure}[h]
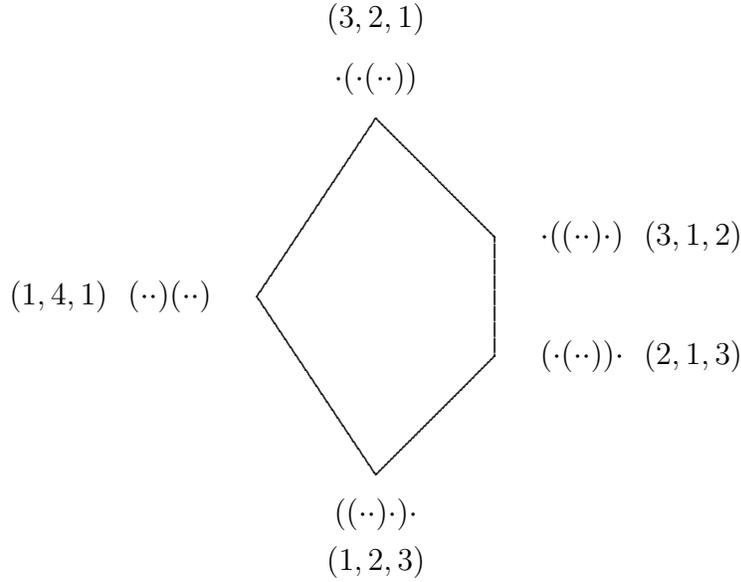

\beginpicture
  \setcoordinatesystem units  <1.5pt,1.5pt>
  \setplotarea x from -150 to 50, y from -30 to 100
\setlinear
\plot 0 0 30 30 30 60 0 90 -30 45 0 0 /
\put{$((\cdot\cdot)\cdot)\cdot$} at 0 -10
\put{$\cdot(\cdot(\cdot\cdot))$} at 0 100
\put{$(\cdot\cdot)(\cdot\cdot)$} at -52 45
\put{$\cdot((\cdot\cdot)\cdot)$} at 52 60
\put{$(\cdot(\cdot\cdot))\cdot$} at 52 30
\put{$(1,2,3)$} at 0 -22
\put{$(2,1,3)$} at 80 30
\put{$(3,1,2)$} at 80 60
\put{$(1,4,1)$} at -80 45
\put{$(3,2,1)$} at 0 115
\endpicture
\caption{\label{fig:Loday_coord} Loday's coordinates of the $\text{Asso}_3$}
\end{figure}
\end{center}
\end{example}

\section{Main result}

Example~\ref{ex:vertices_T3} shows that some vertices of $\text{Perm}_n$ and $\text{Asso}_n$ are in common. We are interested in the intersection of these two polytopes.

\begin{theorem}
$\text{Perm}_n\cap \text{Asso}_n=(\ldots (1\star 2)\star 3)\star \ldots) \star n$ is the maximal never-middle Condorcet domain of size $2^{n-1}$.
\end{theorem}

\begin{proof}
As seen in Example~\ref{ex:vertices_T3}, for $n = 3$ the common vertices of the two polytopes are the points corresponding to the permutations belonging to maximal Condorcet domain
\[
\cd_{3,2}=\{123, 213, 312, 321\}=(1\star 2)\star 3.
\]
We see that $3$ can never be in the middle, hence the condition $3N_{\{1,2,3\}}3 $ is satisfied.
Inductively, suppose that 
\begin{equation}
\label{ind_hyp}
\text{Perm}_{n-1}\cap \text{Asso}_{n-1}=(\ldots (1\star 2)\star 3)\star \ldots) \star (n-1).
\end{equation}

Let $t\in Y_n$ and consider $M_n(t)\in \text{Asso}_n$. Consider the root $v$ of this tree and its weight $w(v)$. We have $w(v)=(i+1)(n-i)$ for some $i\in \{0,\ldots,n-1\}$. If $M_n(t)\in \text{Perm}_n$, we must have $w(v)\le n$. Thus, 
\[
w(v)=(i+1)(n-i)\le n,
\]
from which $i=0$ or $i=n-1$. Thus $t$ is equal to one of the following trees:
\begin{center}
$$\bpic{-50}{50}{0}{80}
\bull{0}{0}
\bull{-70}{70}
\put{$0$} at -70 80
\bull{-21}{21}
\put{$n{-}1$} at 30 80
\bull{70}{70}
\put{$n$} at 70 80
\bull{30}{70}
\put{$t_{-n}$} at -20 60
\setlinear
\plot 30 30 0 0 -50 50 -70 70 /
\plot -21 21 30 70 /
\plot 70 70 30 30 / 
\bull{200}{0}
\bull{130}{70}
\put{$0$} at 130 80
\bull{220}{20}
\put{$1$} at 168 80
\bull{270}{70}
\put{$n$} at 270 80
\bull{168}{70}
\put{$t_{-n}$} at 220 60
\setlinear
\plot 130 70 200 0 270 70 /
\plot 168 70 220 20 /
\epic$$
\end{center}
where $t_{-n}\in Y_{n-1}$ and $M_{t-1}(t_{-n})\in \text{Perm}_{n-1}\cap \text{Asso}_{n-1}$. This means that  
\[
M_n(t)= n \ldots \quad \text{or}\quad M_n(t)=\ldots n,
\]
respectively. We note that the weights for the vertices of $t_{n-1}$ are exactly the same as they were in $t_n$, Hence the induction hypothesis \eqref{ind_hyp} can be used and 
\[
\text{Perm}_{n}\cap \text{Asso}_{n}=(\ldots (1\star 2)\star 3)\star \ldots) \star (n-1)\star n.
\]
It is known \citep{karpov2023symmetric} that such domain is maximal and satisfies for each triple $i<j<k$ the never condition $kN_{\{i,j,k\}}2$ (never-middle one) and has cardinality $2^{n-1}$.
This proves the theorem.
\end{proof}

\begin{note}
The never condition $kN_{\{i,j,k\}}2$ (never-middle one) can be reformulated in terms of pattern avoidance. This will mean that $\text{Perm}_{n}\cap \text{Asso}_{n}$ consists of all permutations of $S_n$ that avoid patterns $132$ and $231$.
\end{note}

\begin{example} 
For $n = 4$ we get the domain
\[
\cd=((1\star 2)\star 3)\star 4=\{1 2 3 4, 2 1 3 4, 3 1 2 4, 3 2 1 4, 4 1 2 3, 4 2 1 3, 4 3 1 2, 4 3 2 1\}.
\]
\end{example}

\bibliographystyle{plainnat}
\bibliography{../cps}

\begin{thebibliography}{11}
\providecommand{\natexlab}[1]{#1}
\providecommand{\url}[1]{\texttt{#1}}
\expandafter\ifx\csname urlstyle\endcsname\relax
  \providecommand{\doi}[1]{doi: #1}\else
  \providecommand{\doi}{doi: \begingroup \urlstyle{rm}\Url}\fi

\bibitem[Danilov and Koshevoy(2013)]{DanilovK13}
V.I. Danilov and G.A. Koshevoy.
\newblock Maximal {C}ondorcet domains.
\newblock \emph{Order}, 30\penalty0 (1):\penalty0 181--194, 2013.

\bibitem[Danilov et~al.(2012)Danilov, Karzanov, and Koshevoy]{DKK:2012}
V.I. Danilov, A.V. Karzanov, and G.A. Koshevoy.
\newblock Condorcet domains of tiling type.
\newblock \emph{Discrete Applied Mathematics}, 160\penalty0 (7-8):\penalty0
  933--940, 2012.

\bibitem[Galambos and Reiner(2008)]{GR:2008}
A.~Galambos and V.~Reiner.
\newblock Acyclic sets of linear orders via the {B}ruhat orders.
\newblock \emph{Social Choice and Welfare}, 30\penalty0 (2):\penalty0 245--264,
  2008.

\bibitem[Karpov and Slinko(2023)]{karpov2023symmetric}
Alexander Karpov and Arkadii Slinko.
\newblock Symmetric maximal condorcet domains.
\newblock \emph{Order}, 40\penalty0 (2):\penalty0 289--309, 2023.

\bibitem[Labb{\'e} and Lange(2020)]{labbe2020cambrian}
Jean-Philippe Labb{\'e} and Carsten~EMC Lange.
\newblock Cambrian acyclic domains: counting c-singletons.
\newblock \emph{Order}, 37\penalty0 (3):\penalty0 571--603, 2020.

\bibitem[Loday(2004)]{loday2004}
Jean-Louis Loday.
\newblock Realization of the {S}tasheff polytope.
\newblock \emph{Arch. Math.}, 83:\penalty0 267--278, 2004.

\bibitem[Puppe and Slinko(2024)]{puppe2024maximal}
Clemens Puppe and Arkadii Slinko.
\newblock Maximal condorcet domains. a further progress report.
\newblock \emph{Games and Economic Behavior}, 145:\penalty0 426--450, 2024.

\bibitem[Slinko(2024)]{slinko2024combinatorial}
Arkadii Slinko.
\newblock A combinatorial representation of arrow's single-peaked domains.
\newblock \emph{arXiv preprint arXiv:2412.05406}, 2024.

\bibitem[Stasheff(1963)]{stasheff1963homotopy}
James~Dillon Stasheff.
\newblock Homotopy associativity of {H}-spaces. {I}, {II}.
\newblock \emph{Transactions of the American Mathematical Society},
  108\penalty0 (2):\penalty0 293--312, 1963.

\bibitem[Tamari(1962)]{tamari1962algebra}
Dov Tamari.
\newblock The algebra of bracketings and their enumeration.
\newblock \emph{Nieuw Arch. Wisk}, 3\penalty0 (10):\penalty0 131--146, 1962.

\bibitem[Ziegler(2012)]{ziegler2012lectures}
G{\"u}nter~M Ziegler.
\newblock \emph{Lectures on polytopes}, volume 152.
\newblock Springer Science \& Business Media, 2012.

\end{thebibliography}

\end{document}